\documentclass[12pt,leqno]{amsart}
\usepackage{amssymb,amsthm,amsmath,latexsym}
\newcommand{\vfi}{\varphi}

\newtheorem{theorem}{\sc Theorem}[section]
\newtheorem{thm}[theorem]{\sc Theorem}
\newtheorem{lem}[theorem]{\sc Lemma}

\newtheorem{cor}[theorem]{\sc Corollary}

\newtheorem{que}[theorem]{\sc Question}
\dedicatory{Dedicated to Professor Antonio Paques on the occasion of his 70th anniversary}
\title[Non-abelian tensor product]{Non-abelian tensor product of residually finite groups}
\author[Bastos]{R. Bastos}
\address{ Departamento de Matem\'atica, Universidade de Bras\'ilia,
Brasilia-DF, 70910-900 Brazil }
\email{bastos@mat.unb.br}
\author[Rocco]{N.\,R. Rocco }
\address{ Departamento de Matem\'atica, Universidade de Bras\'ilia,
Brasilia-DF, 70910-900 Brazil }
\email{norai@unb.br}
\subjclass[2010]{20E26, 20F50, 20J06}
\keywords{Residually finite groups; Locally finite groups; Non-abelian tensor product of groups}

\begin{document}

\maketitle

\begin{abstract}
Let $G$ and $H$ be groups that act compatibly on each other. We denote by $\eta(G,H)$ a certain extension of the non-abelian tensor product $G \otimes H$ by $G \times H$. Suppose that $G$ is residually finite and the subgroup $[G,H] = \langle g^{-1}g^h \ \mid g \in G, h\in H\rangle$ satisfies some non-trivial identity $f \equiv~1$. We prove that if $p$ is a prime and every tensor has $p$-power order, then the non-abelian tensor product $G \otimes H$ is locally finite. Further, we show that if $n$ is a positive integer and every tensor is left $n$-Engel in $\eta(G,H)$, then the non-abelian tensor product $G \otimes H$ is locally nilpotent. The content of this paper extend some results concerning the non-abelian tensor square $G \otimes G$.      
\end{abstract}

\maketitle

\section{Introduction}

Let $G$ and $H$ be groups each of which acts upon the other (on the right), 
\[
G\times H \rightarrow G, \; (g,h) \mapsto g^h; \; \; H\times G \rightarrow
H, \; (h,g) \mapsto h^g
\]
and on itself by conjugation, in such a way that for all $g,g_1 \in G$ and
$h,h_1 \in H$,
\begin{equation}   \label{eq:0}
g^{\left( h^{g_1} \right) } = \left( \left( g^{g^{-1}_1}  \right) ^h \right) ^{g_1} \; \; \mbox{and} \; \; h^{\left( g^{h_1}\right) } =
\left( \left( h^{h_1^{-1}} \right) ^g \right) ^{h_1}.
\end{equation}
In this situation we say that $G$ and $H$ act {\em compatibly} on each other. Let  $H^{\varphi}$ be 
an extra copy of $H$, isomorphic via $\varphi : H \rightarrow
H^{\varphi}, \; h \mapsto h^{\varphi}$, for all $h\in H$. Consider the group $\eta(G,H)$ defined in  \cite{Nak} as 
$$\begin{array}{ll} {\eta}(G,H) =  \langle
G \cup H^{\varphi}\ |  &
[g,{h}^{\varphi}]^{g_1}=[{g}^{g_1},({h}^{g_1})^{\varphi}], \;
[g,{h}^{\varphi}]^{h^{\varphi}_1} = [{g}^{h_1},
({h}^{h_1})^{\varphi}] , \\ & \ \forall g,g_1 \in G, \; h, h_1 \in H
\rangle . \end{array}$$
We observe that when $G=H$ and all actions are conjugations, $\eta (G,H)$ becomes the group $\nu (G)$ introduced in \cite{NR1}: 
$$\begin{array}{ll} {\nu}(G) =  \langle
G \cup G^{\varphi}\ |  &
[g_1,{g_2}^{\varphi}]^{g_3}=[{g_1}^{g_3},({g_2}^{g_3})^{\varphi}] = 
[g_1,{g_2}^{\varphi}]^{g^{\varphi}_3}, \ g_i \in G
\rangle . \end{array}$$

It is a well known fact (see \cite[Proposition 2.2]{Nak}) that the subgroup 
$[G, H^{\varphi}]$ of $\eta(G,H)$ is canonically isomorphic with the {\em non-abelian 
tensor product} $G \otimes H$, as defined by R. Brown and J.-L. Loday in their seminal paper \cite{BL}, the isomorphism being induced by $g \otimes h \mapsto 
[g, h^{\varphi}]$ (see also Ellis and Leonard \cite{EL}). It is clear that the subgroup $[G,H^{\varphi}]$ is normal in $\eta(G,H)$ and one has the decomposition 
\begin{equation} \label{eq:decomposition}
 \eta(G,H) = \left ( [G, H^{\varphi}] \cdot G \right ) \cdot H^{\varphi},
\end{equation}
where the dots mean (internal) semidirect products. In particular, $\nu(G) = ([G,G^{\varphi}] \cdot G) \cdot G^{\varphi}$, where $[G,G^{\varphi}]$ is isomorphic to $G \otimes G$, the non-abelian tensor square of $G$. 

An element $\alpha \in \eta(G,H)$ will be called a {\em tensor} if $\alpha = [a,b^{\varphi}]$ for suitable $a\in G$ and $b\in H$. We write $T_{\otimes}(G,H)$ to denote the set of all tensors.  

Taking to account the actions of $H$ on $G$ and of $G$ on $H$, we denote by $[G,H]$ the subgroup $\langle g^{-1}g^h \mid \ g \in G, h \in H\rangle$ of $G$ and, similarly, $[H,G]$ denotes the subgroup $\langle h^{-1}h^g \mid \ h \in H, g \in G \rangle$ of $H$. In particular, if $G=H$ and all actions are conjugations, then the subgroup $[G,H]$ is the derived subgroup $G'$. If $g \in G$ and $h \in H$, it is customary to write $[g,h]$ rather that $g^{-1}g^h$.      

In the present paper we want to study the following question: If we assume certain restrictions on the set $T_{\otimes}(G,H)$, how does this influence in the structure of the groups $[G, H^{\varphi}]$ and $\eta(G,H)$?  

For elements $x,y$ of an arbitrarily group $G$ we define $[x,{}_1y]=[x,y]$ and, for $i\geq1$, $[x,{}_{i+1}y]=[[x,{}_{i}y],y]$, where $[x,y] = x^{-1}y^{-1}xy$. An element $y\in G$ is called {\em left $n$-Engel} (in $G$) if for any $x\in G$ we have $[x,{}_{n}y]=1$. The group $G$ is called a {\em left $n$-Engel group} if $[x,{}_{n}y]=1$ for all $x,y \in G$.

According to the positive solution of the Restricted Burnside Problem (Zelmanov, \cite{ze1,ze2}) every residually finite group of bounded exponent is locally finite. Following Zelmanov's solution, Wilson proved that if $G$ is a left $n$-Engel residually finite group, then $G$ is locally nilpotent \cite{w}. Later, Shumyatsky established that if $G$ is residually finite and every commutator is left $n$-Engel, then $G'$ is locally nilpotent \cite{shu1}. Another interesting result in this context, also due to Shumyatsky \cite{shu3}, states that if $G$ is a residually finite group satisfying a non-trivial identity and generated by a normal commutator-closed set of $p$-elements, then $G$ is locally finite. A subset $X$ of a group $G$ is called {\it commutator-closed} if $[x,y]\in X$ whenever $x,y\in X$. 

In the article \cite{BRMfM} the authors generalize in a certain way Shumyatsky's result from $G'$ to $[G,G^{\varphi}]$ and $\nu(G)'$, by proving that, given a prime $p$ and a residually finite group $G$ satisfying some non-trivial identity, if every tensor has $p$-power order then  $[G,G^{\varphi}]$ and $\nu(G)'$ are locally finite. In the present paper and under appropriate conditions, we further extend this result to arbitrary non-abelian tensor product of groups $[G,H^{\varphi}]$.    

\begin{thm} \label{thm.1}
Let $p$ be a prime and $G,H$ groups acting compatibly on each other such that the subgroup $[G,H]$ satisfies some non-trivial identity. Assume that $G$ is residually finite and that every tensor has $p$-power order. Then the non-abelian tensor product $[G,H^{\varphi}]$ is locally finite.   
\end{thm}

 In the same paper \cite{BRMfM} we also proved that if $G$ is a residually finite group in which for every $x,y \in G$ there exists a $p$-power $q=q(x,y)$, dividing a fixed power $p^m$, such that $[x,y^{\varphi}]^q$ is left $n$-Engel in $\nu(G)$, then $[G,G^{\varphi}]$ is locally virtually nilpotent. In the present context of groups $G$ and $H$ acting compatibly one on each other, we prove:

\begin{thm} \label{thm.2}
Let $n$ be a positive integer and $G,H$ be groups acting compatibly on each other such that $[G,H]$ satisfies some non-trivial identity. Assume that $G$ is residually finite and that every tensor is left $n$-Engel in $\eta(G,H)$. Then the non-abelian tensor product $[G,H^{\varphi}]$ is locally nilpotent.   
\end{thm}

We do not know whether the hypothesis that $[G,H]$ satisfies some non-trivial identity is really necessary in Theorem \ref{thm.2}. The proof presented here uses this assumption in a very essential way. This suggests the following question: 

\begin{que} \label{question}
Let $n$ be a positive integer and $G,H$ be groups acting compatibly on each other. Assume that $G$ is residually finite and every tensor is left $n$-Engel in $\eta(G,H)$. Is it true that $[G,H^{\varphi}]$ is locally nilpotent? 
\end{que}

The paper is organized as follows. In the next section we collect some structural results to the non-abelian tensor product $[G,H^{\varphi}]$ and related constructions which are later used in the proofs of our main results. In Section 3 we describe some important ingredients of what is often called ``Lie methods in group theory''. Section 4 contains the proofs of the main theorems. 

\section{Preliminary results}

Throughout the rest of the paper we assume that  $G$ and $H$  be groups acting compatibly on each other.

The following result is a consequence of \cite[Proposition 2.3]{BL}.

\begin{lem} 
\label{basic.eta}
The following relations hold in $\eta(G,H)$ for all 
$g, x \in G$ and $h,y \in H$.
\begin{itemize}
\item[$(i)$] $[g, h^{\varphi}]^{[x, y^{\varphi}]} = [g, h^{\varphi}]^{x^{-1}x^y} = [g,h^{\varphi}]^{(y^{-x}y)^{\varphi}}$; 
\item[$(ii)$] $[[g,h^{\varphi}],[x,y^{\varphi}]] = [g^{-1}g^h,(y^{-x}y)^{\varphi}]$.
\end{itemize}
\end{lem}

There is an epimorphism $\lambda:[G,H^{\vfi}] \to [G,H]$, given by $[g,h^{\vfi}] \mapsto g^{-1}g^h \ (= [g,h])$ for all $g \in G$ and $h \in H$. In particular, $[G,H^{\vfi}]/\ker(\lambda)$ is isomorphic to $[G,H]$. The next Lemma is an immediate consequence of Lemma \ref{basic.eta} $(ii)$.

\begin{lem} \label{lem.kernel}
The $\ker(\lambda)$ is a central subgroup of the non-abelian tensor product $[G,H^{\varphi}]$.  
\end{lem}

\begin{lem} \label{lem.closed}
The set $X = \{[g,h] \mid  g \in G, h\in H\}$ is a normal, commutator-closed subset of $G$.    
\end{lem}

\begin{proof}
Note that $X$ is a normal subset of $G$. Choose arbitrarily elements $g,x \in G$ and $h \in H$. Thus $$[g,h]^{x} = x^{-1}(g^{-1}g^{h})x = (g^{-1})^x(g^h)^x = (g^x)^{-1} (g^x)^{h^x} = [g^x,h^x],$$ where $g^x \in G$ and $h^x\in H$. It follows that $[g,h]^x \in X$. 

It remains to show that $X$ is a commutator-closed subset. Choose arbitrarily elements $g,x \in G$ and $h, y \in H$. By Lemma \ref{basic.eta} (ii), $$[[g,h^{\varphi}],[x,y^{\varphi}]] = [g^{-1}g^h,(y^{-x}y)^{\varphi}].$$ In particular, $\lambda([[g,h^{\varphi}],[x,y^{\varphi}]]) = \lambda([g^{-1}g^h,(y^{-x}y)^{\varphi}])$. It follows that the commutator $[[g,h],[x,y]]$ can be rewrite as $[g^{-1}g^h,y^{-x}y]$, where $g^{-1}g^h \in G$ and $y^{-x}y \in H$. Hence $[[g,h],[x,y]] \in X$. 
\end{proof}

When $G=H$ and all actions are conjugations, the epimorphism becomes  $\lambda: [G,G^{\varphi}] \to G'$ given by $[a,b^{\vfi}] \mapsto [a,b]$ for all $a,b \in G$. In particular, $[G,G^{\vfi}]/\ker(\lambda) \cong G'$ and $\ker(\lambda)$ is a central subgroup of $[G,G^{\varphi}]$.

\begin{lem} \label{lem.pe}
Let $n$ be a positive integer, $G$ and $H$ be groups acting compatibly on each other and elements $g \in G$, $h \in H$.  
\end{lem}
\begin{itemize}
\item[(a)] If the tensor $[g,h^{\varphi}]$ has order dividing $n$, then the element $[g,h]$ has order dividing $n$.
\item[(b)] If the tensor $[g,h^{\varphi}]$ is left $n$-Engel in $\eta(G,H)$, then $[g,h]$ is left $n$-Engel in $G$.   
\end{itemize}
\begin{proof}
(a). It is sufficient to see that $$1 =\lambda([g,h^{\varphi}]^n) = [g,h]^n.$$  

\noindent{(b).} By assumption, $[x,_{n}[g,h^{\varphi}]]=1$, for every $x \in G$. Since $[G,H^{\varphi}]$ is normal, it follows that $[x,_{i}[g,h^{\varphi}]] \in [G,H^{\varphi}]$, for every $i \geqslant 1$. Moreover, $$1 = \lambda([x,_{n}[g,h^{\varphi}]]) = [x,_{n}[g,h]],$$ for every $x \in G$ and so $[g,h]$ is left $n$-Engel in $G$. 
\end{proof}

By the decomposition \eqref{eq:decomposition} of $\eta(G,H)$ of which  $[G,H^{\varphi}]$ is a normal subgroup, we have (see also \cite[Theorem 3.3]{NR1}):

\begin{lem} \label{lem.Roc}
The derived subgroup of $\eta(G,H)$ is given by $$\eta(G,H)' = ([G,H^{\vfi}] \cdot G') \cdot (H')^{\vfi}.$$ 
\end{lem}

\section{Associated Lie Algebras}

Let $L$ be a Lie algebra over a field $\mathbb{K}$.
We use the left normed notation: thus if
$l_1,l_2,\dots,l_n$ are elements of $L$, then
$$[l_1,l_2,\dots,l_n]=[\dots[[l_1,l_2],l_3],\dots,l_n].$$
We recall that an element $a\in L$ is called {\it ad-nilpotent} if there exists a positive integer $n$ such that $[l,{}_na]=0$ for all $l\in L$. When $n$ is the least integer with the above property then we say that $a$ is ad-nilpotent of index $n$. 

Let $X\subseteq L$ be any subset of $L$. By a commutator of elements in $X$,
we mean any element of $L$ that could be obtained from
elements of $X$ by means of repeated operation of
commutation with an arbitrary system of brackets
including the elements of $X$. Denote by $F$ the free
Lie algebra over $\mathbb{K}$ on countably many free
generators $x_1,x_2,\dots$. Let $f=f(x_1,x_2,
\dots,x_n)$ be a non-zero element of $F$. The algebra
$L$ is said to satisfy the identity $f=0$ if
$f(l_1,l_2,\dots,l_n)=0$ for any $l_1,l_2,\dots,l_n
\in L$. In this case we say that $L$ is PI. Now, we recall an important theorem of Zelmanov 
\cite[Theorem 3]{zelm} that has many applications in Group Theory.
\begin {theorem}\label{thm.z} 
Let $L$ be a Lie algebra generated by $l_1,l_2,\dots,l_m$. Assume that $L$ is PI and that each commutator in the generators is ad-nilpotent. Then $L$ is nilpotent.
\end{theorem}

Let $G$ be a group and $p$ a prime. In what follows, $$D_i=D_i(G) = \displaystyle \prod_{jp^k \geqslant i} (\gamma_j(G))^{p^k} $$ 
denotes the $i$-th dimension subgroup of $G$ in characteristic
$p$. These subgroups form a central series of $G$
known as the {\it Zassenhaus-Jennings-Lazard series} (this can be found in Shumyatsky \cite[Section 2]{shu2}). Set $L(G)=\bigoplus D_i/D_{i+1}$. 
Then $L(G)$ can naturally be viewed as a Lie algebra 
over the field ${\mathbb F}_p$ with $p$ elements. The subalgebra of $L(G)$ generated by $D_1/D_2$ will be 
denoted by $L_p(G)$. The nilpotency of $L_p(G)$ has strong influence in the structure of a finitely generated group $G$. The following result is due to Lazard \cite{l2}.

\begin{thm}\label{laz} 
Let $G$ be a finitely generated pro-$p$
group. If $L_p(G)$ is nilpotent, then $G$ is
$p$-adic analytic.
\end{thm}

The following result is an immediate corollary of \cite[Theorem 1]{wize}.
\begin{lem}\label{wize}
Let $G$ be any group satisfying some non-trivial identity. Then $L(G)$ is $PI$.
\end{lem}

For a deeper discussion of applications of Lie methods to group theory we refer the reader to \cite{shu2}.

\section{Proofs}

We need the following result, due to Shumyatsky \cite{shu3}. 

\begin{lem} \label{lem.1} 
Let $p$ be a prime and $G$ a residually finite group satisfying some non-trivial identity $f \equiv~1$. Suppose that $G$ is generated by a normal commutator-closed subset $X$ of $p$-elements. Then $G$ is locally finite.
\end{lem}

The proof of Theorem \ref{thm.1} is now easy.

\begin{proof}[Proof of Theorem \ref{thm.1}]

Recall that $p$ is a prime, $G$ and $H$ are groups acting compatibly on each other such that $[G,H]$ satisfies some non-trivial identity $f \equiv~1$. Suppose that $G$ is residually finite and for every $g \in G$ and $h\in H$ there exists a $p$-power $q=q(g,h)$ such that $[g,h^{\varphi}]^q=1$. We need to prove that every finitely generated subgroup of the non-abelian tensor product $[G,H^{\varphi}]$ is finite.   

We first prove that the subgroup $[G,H]$ is locally finite. By Lemma \ref{lem.closed}, the set $\{g^{-1}g^h \mid g\in G, h\in H\}$ is a normal commutator-closed subset in $G$. Since for every $g\in G$ and $h \in H$ the element $g^{-1}g^h$ is a $p$-element and the subgroup $[G,H] \leqslant G$ satisfies some non-trivial identity, it follows that $[G,H]$ is locally finite (Lemma \ref{lem.1}).  

Let $\lambda: [G,H^{\varphi}] \to [G,H]$ given by $\lambda([g,h^{\varphi}]) = [g,h]$. By Lemma \ref{lem.kernel}, $\ker \lambda$ is a central subgroup of $[G,H^{\varphi}]$ and $[G,H^{\varphi}]/\ker(\lambda)$ is isomorphic to $[G,H]$. Let $W$ be a finitely generated subgroup of $[G,H^{\vfi}]$. Since the factor group $[G,H^{\vfi}]/\ker(\lambda)$ is isomorphic to $[G,H]$, it follows that $W$ is central-by-finite. By Schur's Lemma \cite[10.1.4]{Rob}, the derived subgroup $W'$ is finite. Now $W/W'$ is an abelian group generated by finite many elements of finite orders; thus $W/W'$ is also finite and consequently, $W$ is finite. This completes the proof. 
\end{proof}

Our next two results are immediate corollaries in the context of the non-abelian tensor square and related constructions.

\begin{cor} \label{cor.1}
Let $p$ be a prime and $m$ a positive integer.  Let $G$ a residually finite group. Suppose that for every $x,y \in G$ $[x,y^{\vfi}]^{p^m} = 1$. Then the groups $[G,G^{\varphi}]$ and $\nu(G)'$ are locally finite. 
\end{cor}

\begin{proof}
Since every tensor in $\nu(G)$ has order dividing $p^m$, it follows that every commutator in $G$ has order dividing $p^m$. In particular, the group $G$ satisfies the identity $$f = [x_1,x_2]^{p^m} \equiv~1.$$ By Theorem \ref{thm.1}, the non-abelian tensor square $[G,G^{\varphi}]$ is locally finite. Since the derived subgroup $G'$ is an homomorphic image of the non-abelian tensor square $[G,G^{\varphi}]$, it follows that $G'$ is also locally finite. Now, it follows from Lemma \ref{lem.Roc} that $\nu(G)' = ([G,G^{\varphi}] \cdot G') \cdot (G')^{\varphi}$. According to Schmidt's result \cite[14.3.1]{Rob}, $\nu(G)'$ is also locally finite.
\end{proof}

\begin{cor} (\cite{BRMfM})
Let $p$ be a prime and $G$ a residually finite group satisfying some non-trivial identity $f \equiv~1$. Suppose that for every $x,y \in G$ there exists a $p$-power $q=q(x,y)$ such that $[x,y^{\vfi}]^q = 1$. Then the derived subgroup $\nu(G)'$ is locally finite. 
\end{cor}

\begin{proof}
Since $G$ satisfies some non-trivial identity, it follows that $[G,G] = G'$ also satisfies the non-trivial identity. By Theorem \ref{thm.1}, the non-abelian tensor square $[G,G^{\varphi}]$ is locally finite. Since the derived subgroup $G'$ is an homomorphic image of the non-abelian tensor square $[G,G^{\varphi}]$, it follows that $G'$ is also locally finite. The corollary now follows from the result of Schmidt just as in the proof of Corollary \ref{cor.1}. 
\end{proof}

Now we will deal with Theorem \ref{thm.2}: {\it  Recall that $G,H$ are groups that act compatibly on each other, $G$ is residually finite and the subgroup $[G,H]$ satisfies some non-trivial identity. Suppose that every tensor $[g,h^{\varphi}]$ is left $n$-Engel in $\eta(G,H)$ for every $g \in G$ and $h\in H$. We want to prove that the non-abelian tensor product $[G,H^{\varphi}]$ is locally nilpotent.} \\ 

We denote by $\mathcal{N}$ the class of all finite nilpotent groups. For each prime $p$, let $R_p$ be the intersection of all normal subgroups of $G$ of finite $p$-power index. The next lemma is a particular case of \cite[Lemma 2.1]{w}. See also \cite[Lemma 3.5]{shu2}. 

\begin{lem} \label{ws} Let $G$ be a finitely generated residually-$\mathcal{N}$ group. If $G/R_p$ is nilpotent for each $p$, then $G$ is nilpotent.
\end{lem}

We are now in a position to prove Theorem \ref{thm.2}.

\begin{proof}[Proof of Theorem \ref{thm.2}] 

Let $\lambda: [G,H^{\varphi}] \to [G,H]$ the epimorphism given by $[g,h^{\varphi}] \mapsto [g,h]$. Arguing as in the proof of Theorem \ref{thm.1}, we deduce that $\ker(\lambda)$ is a central subgroup of the non-abelian tensor product and the factor group $[G,H^{\varphi}]/\ker(\lambda)$ is isomorphic to the subgroup $[G,H]$. Consequently, it suffices to prove that $[G,H]$ is locally nilpotent.

Let $W$ be a finitely generated subgroup $[G,H]$. Clearly, there exist finitely many elements $[g_1,h_1], [g_2,h_2], \ldots, [g_s,h_s]$ such that $$W \leqslant \langle [g_1,h_1], [g_2,h_2], \ldots, [g_s,h_s] \rangle,$$ where $g_i \in G$ and $h_i \in H$. Set $t_i = [g_i,h_i]$, $i=1,2,\ldots,s$ and $T = \langle t_1,t_2,\ldots,t_s\rangle \leqslant [G,H]$. It is sufficient to prove that $T$ is nilpotent. Since finite groups generated by Engel elements are nilpotent \cite[12.3.7]{Rob}, we conclude that $T$ is a residually-$\mathcal{N}$ group. As a consequence of Lemma \ref{ws}, we can assume that $T$ is residually-$p$ for some prime $p$. Let $L=L_{p}(T)$ be the Lie algebra associated with the Zassenhaus-Jennings-Lazard series $$T = D_1 \geqslant D_2 \geqslant \ldots \geqslant D_r \geqslant \ldots $$ of $T$. Then $L$ is generated by $\tilde{t_i} = t_iD_2$, $i=1,2,\ldots,s$. Let $\tilde{t}$ be any Lie-commutator in the $\tilde{t_i}$'s and $t$ the group-commutator in the $t_i$'s having the same system of brackets as $\tilde{t}$. By Lemmas \ref{lem.closed} and \ref{lem.pe}, every group-commutator of the $t_i$'s is $n$-Engel. It follows that every Lie-commutator $\tilde{t}$ is ad-nilpotent. By assumption, the subgroup $[G,H]$ satisfies the identity $f \equiv~1$. According to Lemma \ref{wize}, $L$ satisfies some non-trivial identity. Now Theorem \ref{thm.z} implies that $L$ is nilpotent. Let $\hat{T}$ denote the pro-$p$ completion of $T$. Then $L_p(\hat{T}) = L$ and thus $L_p(\hat{T})$ is also nilpotent; consequently, $\hat{T}$ is a $p$-adic analytic group by Theorem \ref{laz}. Hence, $T$ has a faithful linear representation over the field of $p$-adic numbers. Since $[G,H]$ satisfies some non-trivial identity, it follows that $T$ does not contains a free subgroup of rank $2$ and so, by Tits' Alternative \cite{tits}, $T$ is virtually soluble. As $T$ is residually-$p$, we have $T$ is soluble. According to Gruenberg's result \cite[12.3.3]{Rob}, the subgroup $T$ is nilpotent, as well. The proof is complete. 
\end{proof}

We have been working under the restriction that $[G,H]$ satisfies some non-trivial identity. In the context of the non-abelian tensor square of groups, this condition appears implicitly by the Engel condition assumption. 
\begin{cor} \label{cor.engel}
Let $n$ be a positive integer. Let $G$ be a residually finite group. Suppose that every tensor is left $n$-Engel in $\nu(G)$. Then the non-abelian tensor square $[G,G^{\varphi}]$ is locally nilpotent. 
\end{cor}

\begin{proof}
Since every tensor is left $n$-Engel in $\nu(G)$, it follows that every commutator is left $n$-Engel in $G$. In particular, $G$ satisfies the identity $$ f = [x_1, _{n}[x_2,x_3]] \equiv~1.$$ According to Theorem \ref{thm.2}, the non-abelian tensor square  $[G,G^{\varphi}]$ is locally nilpotent. 
\end{proof}

We note that the above result gives a positive solution for Question \ref{question} in the context of the non-abelian tensor square.

\end{document}